\documentclass[10pt]{article}
\usepackage{amssymb}
\usepackage{amsmath}
\usepackage{mathtools}
\usepackage{xcolor}

\usepackage{bbold}
\usepackage{enumitem}
\usepackage{xspace}
\usepackage{listings}
\usepackage{amsthm}
\usepackage[english]{babel}
\usepackage{authblk}

\theoremstyle{plain}
\newtheorem{thm}{Theorem}[section]
\newtheorem{corol}[thm]{Corollary}
\newtheorem{lem}[thm]{Lemma}
\newtheorem{claim}[thm]{Claim}

\newtheorem{problem}[thm]{Problem}

\theoremstyle{definition}
\newtheorem{defn}[thm]{Definition}
\newtheorem{conj}[thm]{Conjecture}

\theoremstyle{remark}
\newtheorem*{rem}{Remark}
\newtheorem*{note}{Note}
\newtheorem*{notation}{Notational notes for this section}
\newtheorem{obs}{Observation}
\newtheorem*{fct}{Fact}

\newcommand{\com}[1]{}
\title{Weakening Total Coloring Conjecture: Weak TCC and Hadwiger's Conjecture on  Total Graphs}
\date{\vspace{-5ex}}

\author[1]{Manu Basavaraju}
\affil[1]{Department of Computer Science and Engineering,\protect\\National Institute of Technology Karnataka, Surathkal - 575025, India \protect\\\texttt{manub@nitk.ac.in}}

\author[2]{L. Sunil Chandran}
\affil[2]{Department of Computer Science and Automation,\protect\\Indian Institute of Science, Bangalore - 560012, India  \protect\\\texttt{\{sunil,ankurnaskar\}@iisc.ac.in}}

\author[3]{Mathew C. Francis}
\affil[3]{Computer Science Unit\\Indian Statistical Institute, Chennai Centre,\protect\\Chennai - 600029, India\protect\\\texttt{mathew@isichennai.res.in}}

\author[2]{Ankur Naskar}

\begin{document}

\maketitle

\begin{abstract}

The total graph of a graph $G$, denoted by $T(G)$, is defined on the vertex
set $V(G)\sqcup E(G)$ with $c_1,c_2 \in V(G)\sqcup E(G)$ adjacent whenever $c_1$ and $c_2$ are
adjacent to (or incident on) each other in $G$. The total chromatic number
$\chi''(G)$ of a graph $G$ is defined to be the chromatic number of its total
graph. The well-known Total Coloring Conjecture
or TCC states that for every simple finite graph $G$ having maximum degree
$\Delta(G)$, $\chi''(G)\leq \Delta(G) + 2$. In this paper, we consider two ways to weaken TCC:
\smallskip

\noindent 1. {\itshape Weak TCC:} This conjecture states that for a simple finite graph $G$, $\chi''(G) = \chi(T(G)) \leq\Delta(G) + 3$. While weak TCC is known to be true for 4-colorable graphs, it has remained open for 5-colorable graphs. In this paper, we settle this long pending case.
\smallskip

\noindent 2. {\itshape Hadwiger's Conjecture for total graphs:} We can restate TCC as a conjecture that proposes the existence of a strong $\chi$-bounding function for the class of total graphs in the following way: If $H$ is the total
graph of a simple finite graph, then $\chi(H) \leq\omega(H) + 1$, where $\omega(H)$ is the
clique number of $H$. A natural way to relax this question is to replace
$\omega(H)$ by the Hadwiger number $\eta(H)$, the number of vertices in the largest
clique minor of $H$. This leads to the Hadwiger's Conjecture (HC) for total
graphs: if $H$ is a total graph then $\chi(H) \leq \eta(H)$. We prove that this is true if $H$ is the total graph of a graph with sufficiently large connectivity.

A second  motivation for studying Hadwiger's conjecture for total graphs is the following: Consider the class of split graphs whose vertex set is partitioned into an 
	independent  set $A$ and a  clique $B$, with the  following additional 
	constraints:  (1) Each vertex in $B$ has exactly 2 neighbours in $A$;
	(2) No two vertices in $B$ have the same neighbourhood in $A$. 
	It is known that if Hadwiger's conjecture is proved for the squares of this special  class of split graphs, then it  holds also for the general case.
        Of course, proving the conjecture  for this special case is indeed difficult, and therefore
	it is natural to consider the difficulty level of  Hadwiger's conjecture
for the squares of  graph classes defined by slighly modifying the above class of graphs. 
	A  natural structural  modification is to assume that  {\it both}  $A$ and $B$ are independent sets, keeping everything
	else same. 
	{\it It turns out that  the squares of this modified class of graphs
	is exactly the class of
	total graphs. }   From this perspective, it is not really surprising
	that 
	HC on Total Graphs is also  challenging. 
On the other hand, we   show that weak TCC implies HC on total graphs.
	This perhaps 
	suggests that the
latter is an easier problem than the former.

\end{abstract}


\section{Introduction}

\subsection {Total Coloring Conjecture} 

Let $G$ be a simple finite graph. For a vertex $v\in V(G)$ in $G$, we define $N_{G}(v):=\{ u\in V(G) : u$ is adjacent to  $v$ \text{ in } $G\}$ and $E_{G}(v):=\{ e\in E(G): e $  is incident on $v$ in
$G$ \}. The degree of a vertex $v \in V(G)$ is defined as $d(v) = |N_G(v)| = |E_G(v)|$.  The maximum degree of $G$, $\Delta(G) = \max_{v \in V(G)} d(v)$. 

A vertex coloring is called a \textit {proper vertex-coloring} if no two adjacent vertices are assigned the same color. The minimum number   of colors required to achieve a proper vertex-coloring of
$G$ is called the chromatic number of $G$, and is denoted by $\chi(G)$. 
Like the vertex-coloring problem, the problem of coloring the edges of a graph is also one that has received much attention in the literature. A coloring of the edges of a graph is called a \textit{proper edge-coloring} if no two adjacent edges are  assigned the same color. The minimum number of colors required in any proper edge-coloring of a graph $G$ is called its \textit{edge-chromatic number} or \textit{chromatic index}, and is denoted by $\chi'(G)$.
The following theorem was proved by Vizing :
\begin{thm} [Vizing's Theorem] \label{viz}
For a simple finite  graph $G$, $\Delta(G)\leq \chi'(G)\leq \Delta(G)+1$.
\end{thm}
Vizing's theorem suggests a classification of  simple finite graphs into two classes. A graph $G$ is said to be in  \textit{class I} if $\chi'(G)=\Delta(G)$ and in \textit{class II} if $\chi'(G)=\Delta(G)+1$. 
\medskip

The notion of line graphs allows us to view the edge coloring problem as a vertex coloring problem. Given a graph $G$, the line graph of $G$ is the graph   $L(G)= (V',E')$ where $V'= E(G)$ 
and  $e_1, e_2 \in V'$ form an edge $e_1e_2 \in E'$ whenever $e_1$and $e_2$ are adjacent edges in $G$.  Clearly $\chi'(G) = \chi(L(G))$.  The clique number of $G$, denoted by $\omega(G)$,
is the number of vertices
in a maximum clique of $G$. Clearly $\omega(G) \le \chi(G)$ for all graphs. It is easy to see that $\omega(L(G)) \ge \Delta(G)$, and therefore Vizing's theorem guarantees that the chromatic number of
any line graph  is at most $1$ more than its clique number.  Given a class of graphs ${\cal F}$, if there exists a function $f$ such that $\forall G \in {\cal F},  \chi(G) \le f(\omega(G))$, then
${\cal F}$ is a \emph {$\chi$-bounded class}  and $f$ is a  \emph {$\chi$-bounding function} for ${\cal F}$.  $\chi$-boundedness of graph classes is an extensively studied topic, see the survey by Scott and
Seymour \cite {scottseym} for further
references. Thus Vizing's Theorem guarantees a very strong $\chi$-bounding function for the class of line graphs.

After  vertex and edge colorings it was natural for graph theorists to  consider coloring vertices and edges simultaneously. 
A \textit{total coloring} of a graph $G$ is a coloring of all its \textit{elements}, that is, vertices and edges, such that no two adjacent  
elements are assigned the same color. 
(When we refer to two adjacent \textit {elements} of a graph, they are either adjacent to each other or incident on each other as the case may be.) 
That is, it is an assignment $c$ of colors to $V(G)\cup E(G)$ such that $c|_{V(G)}$ is a proper vertex-coloring, $c|_{E(G)}$ is a proper edge-coloring, and $c(uv)\notin\{c(u),c(v)\}$ for any edge $uv\in E(G)$. The minimum possible number of colors in any total coloring of $G$ is called the \textit{total chromatic number} of $G$, and is denoted by $\chi''(G)$.

The \textit{total coloring conjecture} (or \textit{TCC}) was proposed by Behzad \cite{behzad} and Vizing \cite{vizing} independently between 1964 and 1968.
\begin{conj}[Total coloring conjecture]\label{tcc} For any simple, finite  graph $G$,
    \[ \chi''(G) \leq \Delta(G) + 2  \]
\end{conj}

\begin{rem}
	If we use multigraphs instead of simple graphs, the above statement may not hold.  This is true of Theorem \ref {viz} also. In this paper we consider only simple finite graphs.
\end{rem}

Even though many researchers have examined \textit{TCC} over the years, it remains unsolved till date and is considered one of the hardest open problems in graph coloring. Bollob\'as and Harris \cite{bh} proved that $\chi''(G)\leq \frac{11}{6}\Delta(G)$, when $\Delta(G)$ is sufficiently large. Later Kostochka \cite{kostochka3} proved that $\chi''(G)\leq \frac{3}{2}\Delta(G)$, when $\Delta(G)\geq 6$. The best known result to date for the general case was obtained by Molloy and Reed \cite{molloyreed} :
\begin{thm}\label{molloy}
$\chi''(G)\leq \Delta(G)+C$, where $C$ is equal to $10^{26}$.
\end{thm}
\begin{rem}
	Molloy and Reed have mentioned in \cite{molloyreed} that though they only show a value $10^{26}$ for the constant $C$, with much more effort it can be brought down to $500$.
\end{rem}
While \textit{TCC} remains unsolved for the general case, it is known to hold for some special classes of graphs. For example, it is easy to see that \textit{TCC} is true for all complete graphs and bipartite graphs \cite{bcc}. Another case is the class of graphs with maximum degree at most $5$. 
\begin{thm}[{\cite{rosenfeld}, \cite{vijay}, \cite{kostochka2}, \cite{kostochka}}]\label{tcd}
\textit{TCC} holds for all graphs $G$ with maximum degree at most $5$.
\end{thm}
For planar graphs, the total coloring conjecture is known to hold except when the maximum degree is $6$.

\begin{thm}[{\cite{borodin}, \cite{yap}, \cite{sanderzhao}}]\label{planar}
Let $G$ be a planar graph. If $\Delta(G)\neq 6$, then \textit{TCC} holds for $G$, otherwise $\chi''(G)\leq \Delta(G)+3$.
\end{thm}

Considering the difficulty level of TCC, it makes sense to study relaxations of TCC. 
As the title of the paper indicates, we intend to consider two ways to weaken TCC.  First one is the most obvious way: Increase the upper bound appearing in the statement of TCC. This leads to
\textit {weak-TCC} and $(k)$-TCC; see Section \ref {sec:hadwigertcc} for more details.  Now we will develop  the background for the second one.

Just like we defined line graphs and used it to perceive the edge coloring problem on $G$ as a  vertex coloring problem on $L(G)$, we can define a similar structure in the context of total coloring also.
Given a graph $G$ $=(V,E)$ the \textit{total graph} of $G$ is the graph $T(G)=(V'', E'')$ where $V''=$ $V\cup E$, the set of all elements of $G$, and $c_{1}c_{2}$ $\in E''$ whenever elements $c_1$,$c_2$ are adjacent in $G$. Given a class $\mathcal{F}$ of graphs, we define $T(\mathcal{F}):=$ $\{ T(G): G\in \mathcal{F}\}$. We use ${\cal T}$ to denote the class of total graphs i.e. ${\cal T} = T({\cal G})$, where ${\cal G}$ is the class of all simple finite graphs.
\begin{note} 
We shall call vertices in $T(G)$ that correspond to vertices in the original graph $G$ as \textit{v}-vertices and vertices that correspond to edges in $G$ as \textit{e}-vertices. For ease of notation, we denote an element of $G$, that is, an edge or a vertex in $G$, and its corresponding vertex in $T(G)$, by the same letter. Thus, the \textit{v}-vertex in $T(G)$ corresponding to $x\in V(G)$ is also denoted by $x$.
\end{note}

\begin{obs}\label{ob1}
Any total coloring of a graph $G$ corresponds to a proper vertex coloring of $T(G)$ and vice-versa. Thus, the total chromatic number of $G$ is equal to the chromatic number of $T(G)$ that is, $\chi''(G) = \chi(T(G))$. 
\end{obs}

\begin{obs}\label{ob2}
	For every vertex $x\in V(G)$, the \textit{v}-vertex corresponding to $x$ and the \textit{e}-vertices corresponding to the edges in $E_{G}(x)$ form a clique of order $d_{G}(x)+1$ in $T(G)$. It follows that there exists a clique of order $\Delta(G)+1$ in $T(G)$.  Therefore, $\Delta(G)+1 \leq \omega(T(G)) \leq \chi(T(G))=\chi''(G)$. 
In fact, it is not difficult to see that  the clique number of the total graph of $G$, $\omega(T(G)) = \Delta(G) + 1$, when $\Delta(G) \ge 2$.
\end{obs}

It is easy to see that if $\Delta(G)\leq 1$, then $\chi''(G)=\omega(T(G))$. Thus, in view of Observation \ref {ob2},   Conjecture \ref {tcc} can be thought to be 
suggesting the existence of a very  strong $\chi$-bounding function  for  the class of total graphs, and can be restated as follows: 

\begin {conj} [Restatement of total coloring conjecture]  \label {restateTCC} 
	$ \forall H \in {\cal T}, \chi(H) \le \omega(H) + 1$
\end {conj}    

The total coloring conjecture 
    was  intensely  studied  by several brilliant graph theorists for  more than  
    50  years, but still remains unsolved and has earned the reputation of being one of the toughest problems in graph coloring. Therefore it makes sense to 
    study relaxed versions  of this question. 
From the structural
    perspective, one obvious way to relax Conjecture  \ref {restateTCC}  is to replace clique number $\omega(H)$  by Hadwiger number $\eta(H)$, the number of vertices in the largest clique minor of $H$. This directly leads us to the
    study of Hadwiger's conjecture on Total Graphs, hereafter abbreviated as ``HC on ${\cal T}$''. (See next section for the formal statement of Hadwiger's Conjecture.)

    \begin {problem} [Hadwiger's Conjecture on Total Graphs or ``HC on $\mathcal {T}$''] \label {hconT}
    $$      \chi(H) \le \eta(H), \forall H \in {\cal T} $$
    \end {problem}

\begin {rem}
    Careful readers may object that if  we replace $\omega(H)$ by
    $\eta(H)$ in Conjecture \ref {restateTCC},  the new inequality should be $\chi(H) \le \eta(H) + 1, \forall H \in {\cal T}$. But we will show in Theorem \ref {th1} that Conjecture 
    \ref {restateTCC} indeed implies \textit {HC on $\mathcal T$}  as stated in Problem \ref {hconT}. 
    \end {rem}

    Needless to say, Hadwiger's conjecture is an even more celebrated and long standing conjecture than TCC in 
    graph theory, owing its origin to the Four Color Theorem itself. While the general case of Hadwiger's conjecture remains  unsolved, it was proved for several special classes  of graphs. For example,
	Reed and Seymour \cite{reed} proved Hadwiger's conjecture for line graphs more than 15 years ago. Chudnovsky and Fradkin \cite{chud}  even generalized the result to quasi-line graphs.
    But to the best of our knowledge,  for the class of total graphs, whose definition is similar in spirit to that of line of graphs, Hadwiger's conjecture
    has not yet been proved or studied.

\subsection {Hadwiger's Conjecture: The connection between the general case and  the case of  total graphs} 

Given an edge $e$ $= uv$ in $G$, the \textit{contraction} of the edge $e$ involves the following : deleting vertices $u$ and $v$, introducing a new vertex $w_{e}$, and making $w_{e}$ adjacent to all vertices in the set $(N_{G}(u)\cup N_{G}(v))\setminus\{ u,v\}$. The new graph thus obtained is denoted by $G/e$.

A graph $H$ is said to be a \textit{minor} of $G$ if a graph isomorphic to $H$ can be obtained from $G$, by performing a sequence of operations involving only vertex deletions, edge deletions, and edge contractions.
If $G$ contains $H$ as a minor, then we write $H\preceq G$.

The celebrated \textit{Hadwiger's conjecture} is a far-reaching generalization of the \textit{Four Color Theorem}. It was proposed by Hugo Hadwiger \cite{had} in 1943.
\begin{conj}[Hadwiger's Conjecture]
Given any graph $G$ and $t>0$,
\[ \chi(G) \geq t \implies K_{t} \preceq G \]
In other words, every graph has either a clique minor on $t$ vertices ($K_{t}$-minor) or a proper vertex-coloring using $(t-1)$ colors.       
\end{conj}      

For $t\leq 3$, Hadwiger's conjecture is easy to prove. The $t=4$ case was proved by Hadwiger himself in \cite{had}. The Four Color Theorem was proved by Appel and Haken \cite{appel,appel2} in 1977. Using a result of Wagner~\cite{wagner}, it can be shown that the Four Color Theorem is equivalent to Hadwiger's conjecture for $t=5$. Robertson, Seymour, and Thomas \cite{rst} proved in 1993 that Hadwiger's conjecture holds true for $t=6$. The conjecture remains unsolved for $t\geq 7$. So far, Hadwiger's conjecture has been proved for several classes of graphs; see \cite{belkale}, \cite{chud}, \cite{li}, \cite{reed}, \cite{wood}, \cite{xu}.

%Though we may claim that  the corollary of Theorem \ref {th2}  proves \textit {HC on $\mathcal {T}$} for a significant fraction of the class of total graphs,   our failure to get a complete proof for
%\textit {HC on $\mathcal T$} is conspicuous. Even this much (i.e. corollary \ref {th2}) 
%relies heavily on the 
%difficult result  of Molloy and Reed \cite {molloyreed}. 

 Now we describe  a curious observation from \cite {chandran}.  
Given a graph $G$, its square $G^{2}$ is defined on the vertex set $V(G)$ with $u$ and $v$ being adjacent in $G^2$ whenever the distance between $u$ and $v$ in $G$ is at most $2$. For a class of graphs ${\cal F}$, let ${\cal F}^2$ denote the  set of graphs $\{ G^2 : G \in {\cal F} \}$.  Recall that a split graph is a graph whose vertex set can be partitioned into an independent set and a clique.
Let ${\cal S}$ denote  the special class of split graphs, whose vertex set is  partitioned into an independent set $A$ and a clique $B$ with the following  extra constraints: (1) Each vertex in $B$ has exactly
$2$ neighbours in $A$.  (2) There  are no two vertices in $B$ having the same neighbourhood in $A$.  

 The following fact is from \cite {chandran}.  (In \cite {chandran}, the following fact is not stated explicitly  in this detail, but can be easily read out from the proof of  Theorem 1.2, therein.) 
 
 \begin {fct}
  Proving the general case of Hadwiger's conjecture is equivalent to proving  HC for the class ${\cal S}^2$.
 \end {fct}

 Obviously proving HC for ${\cal S}^2$ is extremely difficult, despite its tantalizingly specialized appearance.   From  this perspective, it  is natural  to
 consider classes of graphs that can be obtained by simple structural  modifications of ${\cal S}$ and see how difficult it is to prove HC for the squares of such classes.  
 One  such  modification
 is to assume that both $A$ and $B$ are independent sets, keeping everything else same, i.e.  the class of bipartite graphs with parts $A$ and $B$ with the extra constraints: (1) Each vertex in $B$ has exactly $2$ neighbours in $A$. (2) No two vertices  in $B$ have the same neighbourhood in $A$.   In other words get a bipartite graph
 from each split graph in ${\cal S}$ by converting the clique $B$  to an independent set.  Let this new class be  denoted by ${\cal \hat S}$.  How difficult is it to prove HC for
 ${\cal \hat S}^2$?   Careful inspection reveals that ${\cal \hat S}$ is a familiar class of graphs, the sub-divided graphs.

%It was remarked in \cite{chandran} that proving Hadwiger's conjecture for the squares of even highly specialized graph classes seems to be challenging. The squares of $2$-degenerate graphs were of particular interest to the authors of \cite{chandran} since $2$-degenerate graphs are close generalizations of $2$-trees, and it looked promising that an approach similar to the one used in \cite{chandran} to prove the conjecture for squares of $2$-trees can be applied to the squares of $2$-degenerate graphs too. But it turned out to be much more difficult than expected. 

%A particularly simple case of $2$-degenerate graphs is the class of subdivided graphs.

A graph $G$ is a subdivided graph if it can be obtained from another graph $H$ by subdividing each edge $uv$ in $H$; that is, replacing $uv$ by a path $uwv$, where $w$ is a new vertex. 
Now if we consider the set of newly introduced vertices as $A$ and the original vertices of $H$ as $B$, it is easy to see that $G \in  \cal \hat S$.  The converse is also true: it is easy
to verify that if a graph  $G$ belongs to ${\cal \hat S}$, then there exists another graph $H$ such that $G$ is obtained by sub-diving all the edges of $H$. 
Interestingly,   the class of squares of subdivided graphs is the same as the class of total graphs.

\begin{fct}
	If $H$ is a graph and $G$ is obtained by subdividing every edge of $H$, then the graph $G^{2}$ is isomorphic to the total graph $T(H)$ of $H$. In other words, the class of
	total graphs ${\cal T} = {\cal \hat S}^2$.
\end{fct}

%Our efforts to prove HC on Total graphs forces us to repeat the opinion from \cite {chandran}:  proving Hadwiger's conjecture for the squares of even highly specialized 
%graph classes seems to be challenging. 

Considering the structural closeness of  ${\cal \hat S}$ and ${\cal S}$ and the fact from \cite {chandran} that HC on ${\cal S}^2$ is equivalent to the general case of HC,  it is not very surprising
that  proving Hadwiger's Conjecture for total graphs turns  out be challenging.

\begin{rem}
Let $G$ be a graph and let $T(G)$ be its total graph. Denote the set of \textit{v}-vertices and the set of \textit{e}-vertices in $T(G)$  by $V$ and $E$, respectively. Note that the subgraph of $T(G)$ induced by $V$ is isomorphic to the original graph $G$, the subgraph of $T(G)$ induced by $E$ is isomorphic to the line graph $L(G)$ of $G$, and the bipartite graph induced by the edges between the sets $V$ and $E$ is precisely the subdivided graph $S(G)$ of $G$. 
\end{rem}

\subsection{Our contributions to Hadwiger's conjecture on  total graphs}

In Section~\ref{sec:hadwigertcc} and Section~\ref{sec:weaktcc},  we   explore the difficulty level of   \textit {HC on $\mathcal {T}$ }  in comparison with  \textit{TCC}.
Our first result is Theorem~\ref{th1} which states that if \textit{TCC} is true for a class $\mathcal{F}$ of graphs that is closed under taking subgraphs, then Hadwiger's conjecture holds true for the class $T(\mathcal{F})$.  Taking $\mathcal {F} = \mathcal {G}$, the class of simple finite graphs, we infer that TCC implies HC  on  $\mathcal {T}$, thus justifying our claim that Problem  \ref {hconT} is a relaxation  of Conjecture \ref {tcc} and its restatement,  Conjecture \ref {restateTCC}. But is TCC indeed a stronger statement than \textit { HC on $\mathcal {T}$} in the sense that
some  hypothesis weaker  than  \textit {TCC}  implies \textit { HC on $\mathcal {T}$}? Our next result confirms this.

We will refer to the statement ``$\forall G \in \mathcal {G}$, $\chi''(G)\leq \Delta(G)+3$'' as   \textit{weak TCC}.
In Theorem~\ref{tcc3}, we show that if  \textit{weak TCC} is true for a class $\mathcal{F}$ of graphs that is closed under taking subgraphs, then Hadwiger's conjecture is true for the class $T(\mathcal{F})$. Thus the weaker hypothesis \textit {weak TCC} implies \textit {HC on $\mathcal {T}$}; therefore \textit {HC on $\mathcal T$ } is strictly  easier than \textit {TCC}. A graph $G$ is $t$-total
critical if $\forall e \in E(G), \chi''(G - e) < \chi''(G) = t$.  We show that
for every $(\Delta(G)+3)$-total critical graph  $G$, $T(G)$ has a clique minor of order $\Delta(G) + 3$. 
%While the branch sets used to construct this minor are intuitive, we could not find any natural way to 
%construct a clique minor of $T(G)$ of order $\Delta(G) + k$ for  $(\Delta(G) + k)$-total critical graphs $G$, for $k \ge 4$. 

For each fixed positive integer $k \ge 2$, let    $(k)$-\textit{TCC} be the statement,   `` $\forall G \in \mathcal G$, $\chi''(G)\leq \Delta(G)+k$.''  Thus  $2$-\textit {TCC} is same as \textit {TCC}
and $3$-\textit{TCC} is same as \textit {weak-TCC}.  Can we prove that the weaker hypothesis $k$-\textit{TCC}  (for some fixed positive integer $k \ge 4$) implies \textit {HC on $\mathcal T$}?
We do not know the answer, but our approach of showing the existence of a $\Delta(G)+k$ clique minor in $T(G)$, assuming that $G$ is $(\Delta+k)$-total critical, does not seem to work when $k\geq 4$. More sophisticated approaches may be needed to answer this question.

\com{ A typical
approach may be to consider an appropriate  third parameter, prove the conjecture directly for graphs having this parameter value above a threshold and  if the value of the parameter is  low,
then find some different method to attack the problem on such graphs.
 Connectivity is usually a natural choice for such a third 
parameter since if the connectivity is low, minimum separators can be used to break the graph into smaller components, one of which may have its chromatic number sufficiently close to
the chromatic number of the original graph. Moreover,
the reason for not being to able find larger minors in $T(G)$ seems to be directly related to the low edge connectivity of $G$. 
}

However, if we have the additional guarantee that the graph $G$ has a high vertex connectivity, then Hadwiger's Conjecture will hold for $T(G)$ if $(k)$-TCC is true for $G$.
We show in Theorem~\ref{th2} that if $(k)$-\textit{TCC} is true, then Hadwiger's conjecture is true for $T(\mathcal{F})$, where $\mathcal{F}:=\{G : \kappa(G)\geq 2k-1 \}$ (here, $\kappa(G)$ denotes the vertex-connectivity of $G$). By combining this with the upper bound of Molloy and Reed \cite{molloyreed},  we get that there exists a constant $C'$  such that Hadwiger's conjecture is true for $T(\mathcal{F})$, where $\mathcal{F}:=\{G : \kappa(G)\geq C' \}$. Here  $C'=2C-1$, where $C$ is the constant from Theorem \ref {molloy};
in \cite {molloyreed} Molloy and Reed  proved $C= 10^{26}$, but mention that
with more detailed analysis, $C$ can be brought down to $500$. 

\com{Unfortunately we could not make much progress for the other side, namely on  what could be told about the total graphs of graphs whose connectivity is low: to begin with, total chromatic number
poses more challenges compared to usual chromatic number when we try to infer about its behavior in the resulting components as we decompose the graph based on its minimum separator. 
}

%In brief, in this paper we prove Hadwiger's conjecture for total graphs except for the total graphs of graphs with small connectivity (bounded by a constant). In fact, our Theorem~\ref{th2} fails only for graphs $G$ that does not have a highly connected induced subgraph $H$ (that is, $\kappa(H)\geq C'$) such that $\Delta(H)=\Delta(G)$. We leave this part as an open problem.  

\subsection{Our contributions to the total coloring literature}

  In this section we concentrate on the first (and the more straightforward)  relaxation of TCC  mentioned earlier: \textit {weak-TCC}. 

 The first reason that motivated us to go through the known literature on  total coloring to find out the important special classes $\mathcal F$  of graphs for which \textit {weak TCC} is known
 to hold  (while \textit {TCC} may still remain unsolved), is the promise of 
 Theorem \ref {tcc3}  that  \textit {HC}  is true for  $T(\mathcal F)$ for such $\mathcal F$. We realized that there are some important graph classes in this category,
 for example  planar graphs.
% In fact \textit {weak TCC} is known to hold for the class of $4$-chromatic graphs, whereas \textit {TCC} seems be  too tough for this class.
Unfortunately, for general graphs, it seems the chance of either \textit{TCC} or weak \textit{TCC} getting proved in the near future is very less. 
%In fact, the best known result to date for the general case
%was obtained by Molloy and Reed~\cite{molloyreed}, who proved that $\chi''(G)\leq \Delta(G)+C$, where $C$ is equal to $10^{26}$.

%As discussed in the previous section, while there is a natural way to construct a clique minor of order $\Delta(G) + 3$ in $T(G)$ for any $G \in \mathcal G$, any bigger clique minor is not
%easy to find. It is very likely that there are total graphs $T(G)$  without any clique minors of  order $\Delta(G) + 4$ or more. This structural intuition is a second reason for us to consider
%     \textit {weak-TCC}
%    with an interest that goes beyond the view that it  is a weaker version of \textit {TCC}.  

Is it possible that 
 \textit {weak TCC}  is  a more reasonable conjecture than TCC? It is  remarkable that another  
 very well-known and widely believed conjecture, the \textit{list coloring conjecture}, implies the weak \textit{TCC}, whereas  \textit {TCC} seems beyond its reach. 
\textit{TCC} is indeed a bold conjecture, in the sense that at times it makes us think that it is not unreasonable to look for counterexamples. For example, despite intense research for decades, the conjecture could not be proved for planar graphs with maximum degree $6$, though it has been found to hold good for all other cases of planar graphs. On the other hand, it is  very easy to prove weak \textit{TCC} for all planar graphs: \textit{If $G$ is a planar graph, then $\chi''(G)\leq \Delta(G)+3$}. Moreover, \textit {weak TCC}  is in fact true for a  much wider class, namely $4$-colorable graphs, which properly includes
planar graphs, whereas proving TCC on $4$-colorable graphs seems to be  extremely  difficult.  This led us to survey the status of weak TCC on $5$-colorable graphs, and try  to fill the research
gap therein.

Weak \textit{TCC} can be proved for $4$-colorable graphs (that is, when $\chi(G)\leq 4$) using the following well-known argument: We color the vertices of the graph using the colors $\{1, 2, 3, 4\}$. Applying Vizing's theorem we can color the edges of the graph using $\Delta+1$ colors from $\{3, 4, \ldots,\Delta(G) + 3 \}$. We uncolor all edges with colors $3$ or $4$. Note that for every uncolored edge, there exist at least two colors in $\{1,2,3,4\}$ that are not the colors assigned to its endpoints. Let us associate with each uncolored edge the list containing these two colors. It is easy to see that the subgraph induced by the uncolored edges consists only of paths and even cycles, and is therefore $2$-edge-choosable. This proves weak \textit{TCC} for $4$-colorable graphs. As a special case, planar graphs satisfy weak \textit{TCC} since by the four color theorem all planar graphs are $4$-colorable.

As mentioned earlier this naturally leads to the question whether weak \textit{TCC} can be proved for $k$-colorable graphs with $k\geq 5$. It is not the case that
 researchers who worked on total coloring  failed  to notice this question altogether. 
In fact, from the early days of research in total coloring, researchers have tried to find upper bounds for total chromatic number of a graph in terms of its (vertex) chromatic number. 

\begin {rem} 
If we go through the mainstream literature on
total coloring, apart from efforts to  bring the upper bound closer to the conjectured $\Delta +2$, there were efforts to study TCC for graph classes defined by their maximum degree, i.e. $\Delta \le 4$,
$\Delta \le 5$ etc. Obviously the other two parameters for such study were  the chromatic number $\chi(G)$ and the chromatic index $\chi'(G)$. The latter is essentially $\Delta$, due to Vizing's Theorem.
Studying TCC for graphs of bounded chromatic number  was  the next natural option after bounded maximum degree graphs. 
\end {rem} 

The most important result in this direction was proved by Hind \cite{hind}.
\begin{thm}[Hind]\label{hind}
 For every graph $G$, \[\chi''(G)\leq\chi'(G)+2\left\lceil\sqrt{\chi(G)}\right\rceil\leq  \Delta(G)+1+2\left\lceil\sqrt{\chi(G)}\right\rceil.\]
\end{thm}

We note that for small values of $\chi(G)$ (such as $\leq 9$), Theorem \ref{hind} implies $\chi''(G)\leq \Delta(G)+7$. Naturally many researchers have attempted to improve this result using adaptations of the technique used in Hind's proof. 
One such result was presented by S\'anchez-Arroyo \cite{sanchez} and a slightly better bound was obtained by Chew \cite{chew}.

\begin{thm}[Chew]\label{chew}
For any connected multigraph $G$, 
\[ \chi''(G)\leq \begin{cases}
    \chi'(G) + \left\lceil\frac{\chi(G)}{3}\right\rceil+ 1 & \mbox{if } \chi(G) \equiv 2 \pmod{3} \\[.1in]
    \chi'(G) + \left\lfloor\frac{\chi(G)}{3}\right\rfloor + 1 & \text{otherwise}
  \end{cases} \] 
\end{thm}
This implies that $\chi''(G)\leq \chi'(G)+3$, for $\chi(G)\leq 5$. Although weak \textit{TCC} for \textit{class-I} $5$-colorable graphs follows from this result, weak \textit{TCC} is not known to hold for the entire class of $5$-colorable graphs till date. In Theorem~\ref{th4}, we prove the following long-pending result:
\begin{center}
    \textit{Weak \textit{TCC} is true for $5$-colorable graphs.}
\end{center}
\begin{rem}
The method used in our proof is completely different from the  proofs of Theorem~\ref{hind} and Theorem~\ref{chew}.
\end{rem}

\begin{note}
Proving \textit{TCC} for $5$-colorable graphs is likely to be a considerably harder problem. In fact, \textit{TCC} remains to be proved even for $4$-colorable graphs. As mentioned before the $\Delta=6$ case for planar graphs is still open even after decades of research. 
\end{note}

\section{Hadwiger's conjecture and total coloring}\label{sec:hadwigertcc}

\begin{defn}[total-critical graph]
A graph $G$ is said to be $t$-\textit{total-critical} if $\chi''(G)$ $=t$ and $\chi''(G- e) \leq$ $t-1$, for any edge $e$ of $G$.
\end{defn}

Note that any graph that contains at least one edge has total chromatic number at least 3.
\begin{lem}\label{lm1}
If a connected graph $H$ is $t$-total-critical, where $t\geq \Delta(H)+2$, then $H$ has no cut-vertices.
\end{lem}

\begin{proof}
Let $H$ be a connected $t$-total-critical graph, where $t\geq \Delta(H)+2$. 
Let us choose a vertex $v$ in $H$ with $d_{H}(v)=d$ (say). We claim that $v$ is not a cut-vertex in $H$. Otherwise if $v$ is a cut-vertex then $2\leq d\leq \Delta(H)$. We can assume that for some $s<d$, $s$ neighbours of $v$ say, $v_{1},v_{2}, \ldots,v_{s}$ lie in one connected component $C_{1}$, while the remaining $d - s$ neighbours, $v_{s+1},\ldots,v_{d}$ lie in the remaining connected components $C_{2}, \ldots ,C_{k}$ ($k\geq 2$) of $H-v$. Let the edge $vv_{i}$ be denoted by $e_{i}$, for all $i\leq d$. Define $T_{1} := H[V(C_{1})\cup\{v\}]$ and $T_{2} := H - V(C_1)$. Both $T_{1}$ and $T_{2}$ can be total colored with $(t-1)$ colors since $H$ is $t$-total-critical.

Consider a total coloring $\varphi: V(T_{1})\cup E(T_1)\rightarrow [t-1]$ of $T_1$. Note that in $\varphi$, the colors on the edges $e_1,e_2,\ldots,e_s$ and the vertex $v$ are all different. We can assume without loss of generality (by renaming the colors if necessary) that $\varphi(v)=1$ and $\varphi(e_{i})=i+1$ for each $i\in [s]$. Now consider a total coloring $\psi : V(T_{2})\cup E(T_2)\rightarrow [t-1]$ of $T_2$. As in the previous case, in $\psi$, the colors on the edges $e_{s+1},e_{s+2},\ldots,e_d$ and the vertex $v$ are all different. Again by renaming colors if necessary, we can assume that $\psi(v)=1$ and $\psi(e_{i})=i+1$ for each $i\in\{s+1,s+2,\ldots,d\}$. Notice that the renaming of colors in both $\varphi$ and $\psi$ is possible because $d+1\leq\Delta(H)+1\leq t-1$. We now construct a total coloring $\phi:V(H)\cup E(H)\rightarrow [t]$ of $H$ by combining the colorings $\varphi$ of $T_1$ and $\psi$ of $T_2$ (i.e. $\phi(v)=1$ and for every other element (vertex or edge) $x$ of $H$, $\phi(x)=\varphi(x)$ if $x$ belongs to $T_1$ and $\phi(x)=\psi(x)$ if $x$ belongs to $T_2$). It is easy to verify that $\phi$ is a total coloring of $H$ using just $t-1$ colors. This contradicts the fact that $H$ is a $t$-total-critical graph.
\end{proof}

\begin{lem}\label{lm2}
Let $G$ be a connected $(\Delta(G) + k)$-total-critical graph on at least $3$ vertices, for some $2<k\leq \Delta(G)$. Then the minimum degree of $G$ is at least $k$.
\end{lem}

\begin{proof}
Assume that there exists a vertex $w\in G$ with $d_{G}(w)\leq k-1$, and let $\Delta(G)=\Delta$. Since $G$ has at least three vertices and there are no cut-vertices in $G$ (Lemma \ref{lm1}), $w$ cannot have degree $1$. 

Let the neighbours of $w$ be $v_{1},v_{2}, \ldots ,v_{j}$, where $2\leq j\leq k-1$. Call the edge $v_{i}w$ as $e_{i}$, for each $i\in [j-1]$ and $v_{j}w$ as $f$. Then $G-f$ has a total coloring using $\Delta+k-1$ colors, which we shall assume to be the set $[\Delta+k-1]$.

Assume that $v_{i}$ is assigned color $\alpha_{i}$ for each $i\in [j]$ and the edge $e_{i}$ is assigned color $\beta_{i}$ for each $i\in [j-1]$. Call the color assigned to $w$ as $c$.

If $\alpha_{j} = c$, pick a color $c'\in \{1, 2,\ldots ,\Delta+k-1\}\setminus(\{\alpha_{i}: i\in [j-1]\}\cup\{ \beta_{i} : i\in [j-1]\}\cup\{c\}$) (such a $c'$ exists because $k\leq \Delta$ implies $\Delta+k-1 > 2(k-2)+1$), and assign it to the vertex $w$.  Now, the vertices $w$ and $v_{j}$ have different colors on them. Else if $\alpha_{j}\neq c$, let $c':=c$.

Since $d_{G}(v_{j})\leq\Delta$, and therefore $d_{G-f}(v_j)\leq\Delta-1$, there exists a list $L$ of at least $(k-1)$ colors that are not assigned to $v_{j}$ or its incident edges. If at least one of the colors in $L$ is not in $\{c'\}\cup\{\beta_{i}: i\in [j-1]\}$, assign it to the edge $f$ giving $G$ a $(\Delta+k-1)$-total coloring. Otherwise since $\lvert L \rvert\geq k-1$ and $w$ has at most $k-2$ incident edges in $G-f$, we have the set of colors that appear on $w$ and its incident edges is exactly the list $L$; that is, $\{c'\}\cup\{\beta_{i}: i\in [j-1]\}=L$. Choose a color $c''\in$ $\{1,2,\ldots,\Delta+k-1\}\setminus (\{c',\alpha_{j}\}\cup\{\alpha_{i}:i\in [j-1]\}\cup\{\beta_{i}:i\in [j-1]\})$ (such a $c''$ exists as $j\leq k-1 < \Delta$ implies $\Delta+k-1 > 2(k-2) +2$), and assign it to the vertex $w$ and then assign color $c'$ to the edge $f$, giving a $(\Delta+k-1)$-total coloring of $G$.

This is a contradiction as $\chi''(G)=\Delta+k$ by hypothesis. Hence, $d_{G}(w)\geq k$ for each vertex $w\in G$.
\end{proof}

\begin{thm}\label{th1}
Let $\mathcal{F}$ be a class of graphs that is closed under the operation of taking subgraphs. If \textit{TCC} is true for $\mathcal{F}$, then Hadwiger's conjecture holds for the class $T(\mathcal{F})$.
\end{thm}  

\begin{proof}
Assume that \textit{TCC} holds for all graphs in $\mathcal{F}$. Let $G\in$ $\mathcal{F}$ and $\Delta :=$ $\Delta(G)$. Therefore by Observation~\ref{ob2}, $\Delta + 1\leq$ $\chi''(G)$ $\leq \Delta + 2$. Note that the last inequality follows because \textit{TCC} holds for $\mathcal{F}$. 
\vspace{3mm}\newline
\textbf{Case 1:} $\chi''(G)$ $=\Delta+1$.
\vspace{1mm}\newline
$T(G)$ contains a clique of size $\Delta+1$ (Observation~\ref{ob2}). Therefore, this case is trivial.
\vspace{3mm}\newline
\textbf{Case 2:} $\chi''(G)$ $=\Delta+2$.
\vspace{1mm}\newline
Let $H$ be a $(\Delta+2)$-total critical subgraph of $G$. We then have $\chi''(H) =\Delta + 2$. It is clear that $\Delta(H)=\Delta$, otherwise since $H$ also belongs to $\mathcal{F}$, \textit{TCC} would imply that $\chi''(H)\leq \Delta + 1$. Choose a vertex $v_{\Delta}\in H$ such that $d_{H}(v_{\Delta})=\Delta$. Let $e_1,e_2,\ldots,e_{\Delta}$ be the edges incident on $v_{\Delta}$ in $H$. From Lemma~\ref{lm1} we know $H-v_{\Delta}$ is connected. This implies that the subgraph $X$ of $T(G)$ induced by the elements of $H-v_{\Delta}$ is connected. In $T(G)$, contract the subgraph $X$ into a single vertex $w$. Since the subgraph $X$ contains the \textit{v}-vertex corresponding to every neighbour of $v_\Delta$, the vertex $w$ is now adjacent to the \textit{e}-vertices $e_{1},\ldots,e_{\Delta}$ and the \textit{v}-vertex $v_{\Delta}$, forming a $(\Delta+2)$-clique minor in $T(G)$.
\end{proof}

\section{Weak total coloring conjecture} \label {sec:weaktcc}

We may state weaker versions of Conjecture~\ref{tcc} by relaxing the upper bound on $\chi''(G)$ as follows :

\begin{conj}[$(k)$-\textit{TCC}]
Let $k\geq2$ be a fixed positive integer. For any graph $G$, \[ \chi''(G)\leq \Delta(G) + k \]
\end{conj}
\medskip

\begin{note}
$(2)$-\textit{TCC} is the same as Conjecture \ref{tcc} ; that is, the total coloring conjecture. We shall refer to $(3)$-\textit{TCC} as the \emph{weak total coloring conjecture} or \emph{weak-\textit{TCC}} for short.
\end{note}
\medskip

We first make a general observation about graphs that do not contain cut-vertices.
\begin{claim}\label{cl1}
Suppose that $G$ is a 2-connected graph and $v$ is a vertex in $G$.
Then there exists a vertex $w\in N(v)$ such that $\{w,v\}$ is not a separator of $G$.
\end{claim}
\begin{proof}
Suppose for contradiction that each set in $\{\{v,w\}\}_{w\in N(v)}$ is a separator of $G$. Since $G$ has no cut-vertices, we can conclude that each set in $\{\{v,w\}\}_{w\in N(v)}$ is a minimum separator of $G$. For $w\in N(v)$, let $F(v,w)$ be the number of vertices in the smallest component in $G-\{v,w\}$. Choose $w\in N(v)$ such that $F(v,w)$ is as small as possible. Let $S$ be the smallest component in $G-\{v,w\}$. Since $\{v,w\}$ is a minimum separator, $v$ has a neighbour in each component of $G-\{v,w\}$, and so also in $S$. Let $x\in S\cap N(v)$. Now consider the graph $G-\{w,x\}$. Since $w$ had a neighbour in each connected component of $G-\{v,w\}$ (as $\{v,w\}$ is a minimum separator), all the components of $G-\{v,w\}$, together with $w$, now form one connected component of $G-\{w,x\}$. Thus the other connected components of $G-\{w,x\}$ must all be subgraphs of $S$. Since there is at least one other connected component and it cannot contain $x\in S$, we can conclude that there is a connected component of $G-\{w,x\}$ that is smaller than $S$. This contradicts the choice of $w$ and $S$.
\end{proof}

\begin{lem}\label{newlemma}
Let $H$ be a $(\Delta+3)$-total-critical graph having maximum degree $\Delta$. Then $T(H)$ contains a clique minor of order $\Delta+3$.
\end{lem}
\begin{proof}
Recall that if $H$ is not connected, then $H$ contains at most one connected component that is not an isolated vertex. Since the maximum degree of this component is also $\Delta$ and it is also $(\Delta+3)$-total-critical, we will be done if we prove the statement of the lemma for this component. So from here on, we assume that $H$ is connected.

Choose a vertex $v_{\Delta}$ in $H$, with $d_{H}(v_{\Delta})=\Delta$.
By Lemma~\ref{lm1}, $H$ does not contain any cut-vertices. Then by Claim~\ref{cl1}, there exists a $w\in N_{H}(v_{\Delta})$, such that $\{v_{\Delta},w\}$ is not a separator in $H$.
Recall that \textit{TCC} holds for all graphs with maximum degree at most $5$ (Theorem~\ref{tcd}). Therefore, $\Delta\geq 6$ since $\chi''(H)=\Delta+3$ by assumption. It follows from Lemma~\ref{lm2} that degree of $w$ in $H$ is at least $3$, that is, there are at least two edges between $w$ and $H\setminus\{v_{\Delta},w\}$.
Let $e'$ and $e''$ be the \textit{e}-vertices in $T(H)$ corresponding to these two edges. 

Our goal is to construct a clique minor of order $\Delta+3$ ($K_{\Delta+3}$-minor) in $T(H)$. For this we use the following \textit{branch sets} :
\begin{enumerate}
\item The set $S^{0}$ consisting of the \textit{v}-vertices that correspond to $\{v_{\Delta},w\}$. Note that $S^{0}$ is connected in $T(H)$ and can be contracted to a vertex $z$.
\item The singleton $\{e\}$ for each \textit{e}-vertex $e\in T(H)$ that corresponds to an edge in $E_{H}(v_{\Delta})$. Let $S^{1}$ denote the set of these $\Delta$ \textit{e}-vertices.
\item  $S^{2}=S'\cup\{ e'\}$, where $S'$ is the set of \textit{e}-vertices corresponding to the edges in $E(H\setminus \{v_{\Delta},w\})$. Note that $S^{2}$ is connected and can be contracted to a single vertex $y$.
\item  $S^{3}=S''\cup\{e''\}$, where $S''$ is the set of \textit{v}-vertices corresponding to the vertices in $V(H\setminus \{v_{\Delta},w\})$. $S^{3}$ is also connected and can be contracted to a single vertex $x$.
\end{enumerate}
Now we show that $S^{1}\cup\{x\}\cup\{y\}\cup\{z\}$ forms the desired clique of order $\Delta+3$ in the resulting graph.

Let the \textit{e}-vertex corresponding to the edge $v_{\Delta}w$ be $f$. First, $(S^{1}\setminus\{f\})\cup\{x\}\cup\{y\}$ forms a clique of order $\Delta+1$ since $H\setminus\{v_{\Delta},w\}$ is connected and contains end-points of all edges incident on $v_{\Delta}$ other than $v_{\Delta}w$.

Now obviously $f$ is adjacent to each vertex in $S^{1}\setminus\{f\}$ and also to $e'$ and $e''$ in $T(H)$. It follows that $f$ is adjacent to $x$, $y$ and each member of $S^{1}\setminus\{f\}$.

Now $z$ is obviously adjacent to all vertices in $S^{1}$. Moreover $z$ is adjacent to $x$, since $e''$ is adjacent to $w$ in $T(H)$, and adjacent to $y$, since $e'$ is adjacent to $w$ in $T(H)$. 
\end{proof}
\begin{thm}\label{tcc3}
Let $\mathcal{F}$ be a class of graphs that is closed under the operation of taking subgraphs. Then, Hadwiger's conjecture holds for the class $T(\mathcal{F})$, if weak \textit{TCC} is true for graphs in $\mathcal{F}$.
\end{thm}
\begin{proof}
Let $G$ be a graph in $\mathcal{F}$. Let $\Delta=\Delta(G)$.
Since weak \textit{TCC} holds for $G$, we know that $\chi''(G)\leq\Delta+3$. Together with Observation \ref{ob2}, this implies that $\chi''(G)\in\{\Delta+1,\Delta+2,\Delta+3\}$.
\vspace{3mm}\newline
\textbf{Case 1:} $\chi''(G)=\Delta+3$
\vspace{1mm}\newline
Take a $(\Delta+3)$-total-critical subgraph $H$ of $G$.
Clearly, $\Delta(H) =\Delta$; otherwise if $\Delta(H)<\Delta$, since $H$ also belongs to $\mathcal{F}$, $H$ would have a $(\Delta+2)$-total coloring by weak \textit{TCC}, which is a contradiction. Now by Lemma~\ref{newlemma}, $T(H)$ contains a clique minor of order $\Delta+3$. As $T(H)$ is a subgraph of $T(G)$, we are done.
\vspace{3mm}\newline
\textbf{Case 2:} $\chi''(G)=\Delta+2$
\vspace{1mm}\newline
Let $H$ be a $(\Delta+2)$-total critical subgraph of $G$. Since $H$, being a subgraph of $G$, belongs to $\mathcal{F}$, and therefore satisfies weak \textit{TCC}, we have that $\Delta(H)\geq\Delta-1$. As observed earlier, we assume that $H$ is connected, since if not, we can just take the only component of $H$ that is not an isolated vertex to be $H$.
If $\Delta(H) = \Delta$, then choose a vertex $v_\Delta$ of $H$ such that $d_H(v_\Delta)=\Delta$. By Lemma~\ref{lm1}, we know that $H-v_{\Delta}$ is connected. Then in $T(H)$, we can contract the $v$-vertices corresponding to vertices in $V(H)-v_\Delta$ to a single vertex $z$ (as in the proof of Theorem~\ref{th1}) so that the \textit{e}-vertices corresponding to the edges incident on $v_{\Delta}$, the \textit{v}-vertex corresponding to $v_{\Delta}$, and the vertex $z$ form a clique of order $\Delta+2$ in the resulting graph.
On the other hand, if $\Delta(H)=\Delta-1$, then since $H$ is $(\Delta(H)+3)$-total-critical, we have by Lemma~\ref{newlemma} that $T(H)$ contains a clique minor of order $(\Delta-1)+3=\Delta+2$.
\vspace{3mm}\newline
\textbf{Case 3:} $\chi''(G)=\Delta+1$
\vspace{1mm}\newline
As we can see from Observation \ref{ob2}, this case is trivial. 
\end{proof}

\begin{corol}
Hadwiger's conjecture holds true for total graphs of all planar graphs.
\end{corol}
\begin{proof}
Follows from Theorems~\ref{planar} and \ref{tcc3}.
\end{proof}

Given a graph $G$ with a list $L_{e}$ of colors assigned to each edge $e$, a \emph{list edge coloring} is a proper coloring of $E(G)$ such that each edge $e$ is assigned a color from the list $L_{e}$. 
A graph $G$ is said to be $k$-\emph{edge-choosable} if a list edge coloring of $G$ exists for any assignment of lists of size $k$ to the edges of $G$. The smallest positive integer $k$ such that $G$ is $k$-edge-choosable is called the \emph{edge choosability} of $G$, which is denoted as $ch'(G)$. The following is a well-known fact.

\begin{lem}\label{ech}
For any graph $G$, $\chi''(G)\leq ch'(G)+2$.
\end{lem}

The \emph{list coloring conjecture} is as follows :
\begin{conj}[List coloring conjecture]
For any graph $G$, $ch'(G)$ $=\chi'(G)$.
\end{conj}

\begin{thm}
Let $\mathcal{F}$ be a class of graphs closed under the operation of taking subgraphs. Hadwiger's conjecture holds for the class $T(\mathcal{F})$, if $\mathcal{F}$ satisfies the list coloring conjecture.
\end{thm}
\begin{proof}
Since $\mathcal{F}$ satisfies the list coloring conjecture, for each $G\in\mathcal{F}$, $\chi''(G)\leq ch'(G)+2=\chi'(G)+2\leq \Delta(G)+3$ (by Lemma~\ref{ech} and Theorem~\ref{viz}), that is, $\mathcal{F}$ satisfies weak \textit{TCC}. The theorem then follows from Theorem \ref{tcc3}.
\end{proof}

Thus, the validity of the list coloring conjecture would imply that Hadwiger's conjecture is true for all total graphs.

\section{Connectivity and Hadwiger's conjecture for total graphs}\label{sec:connhadwiger}
Given a partition of the vertex set of a graph, an edge of the graph is said to be a \emph{cross-edge} if its endpoints belong to different sets of the partition.

Our next result makes use of the following well-known theorem :  
\begin{thm}[Tutte; Nash-Williams; See~\cite{diestel}]\label{nw}
A multigraph contains $k$ edge-disjoint spanning trees if and only if every partition $P$ of its vertex set contains at least $k(|P| -1)$ cross-edges.
\end{thm}
\begin{corol}[See~\cite{diestel}]\label{nwc}
Every $2k$-edge-connected multigraph $G$ has $k$ edge-disjoint spanning trees.
\end{corol}

\begin{thm}\label{th2}
Let $G$ be a $(2k-1)$-connected graph. Then, the total graph $T(G)$ contains a clique minor of size $\Delta + k$, where $\Delta$ $=\Delta(G)$.
\end{thm}

\begin{proof}
Let $v_{\Delta}$ be a vertex of maximum degree in $G$. Let the set of \textit{e}-vertices in $T(G)$ corresponding to the edges in $E_{G}(v_{\Delta})$ be denoted by $Z$. Note that $Z$ induces a $\Delta$-clique (call it $K_{Z}$) in $T(G)$. Let $H$ be the graph obtained from $G$ by removing $v_{\Delta}$. 
We will show that there exist $k$ disjoint but pair-wise adjacent connected induced subgraphs of $T(H)$, so that by contracting these induced subgraphs we get $k$ more vertices that can be added to the existing $\Delta$-clique $K_{Z}$ thus achieving the desired $(\Delta+k)$-clique minor in $T(G)$.

Since $G$ is $(2k-1)$-connected, the vertex-connectivity of $H$ is at least $2k-2$. It follows that, edge-connectivity of $H$ is also at least $2k-2$. Hence, by Corollary \ref{nwc}, $H$ has $(k-1)$ pair-wise edge-disjoint spanning trees.

Let $T_{1},\ldots ,T_{k-1}$ be the pair-wise edge-disjoint spanning trees in $H$. Let $E_{i}'$ be the set of \textit{e}-vertices in $T(G)$ corresponding to the edges in $E(T_{i})$, for $i\in [k-1]$. Clearly the induced subgraph of $T(G)$ on $E_{i}'$ is isomorphic to the line graph of $T_{i}$ and is connected and therefore, can be contracted to a single vertex, say $y_{i}$. Let $Y:=\{y_{i} : i\in [k-1]\}$. The set of \textit{v}-vertices of $T(G)$ is also connected since they induce a subgraph isomorphic to $H$ itself, and can be contracted to a single vertex $v^{*}$. Let $F$ be the graph so obtained. Since $T_1,\ldots,T_{k-1}$ are all spanning trees of $H$, it follows that $y_iy_j\in E(F)$ for distinct $i,j\in [k-1]$ and also that $v^*y_i\in E(F)$ for all $i\in [k-1]$. Further, it can be seen that in $F$, each $y_{i}\in Y$ and $v^{*}$ are adjacent to all the \textit{e}-vertices in $Z$. Therefore, $Y\cup Z\cup \{v^{*}\}$ forms a clique of order $\Delta+k$ in $F$, implying that $T(G)$ contains a $(\Delta+k)$-clique minor.
\end{proof}

\begin{corol}
If a $(2k-1)$-connected graph $G$ satisfies $(k)$-\textit{TCC}, then Hadwiger's conjecture is true for $T(G)$.
\end{corol}
\begin{proof}
Since $G$ satisfies $(k)$-\textit{TCC}, we have $\chi(T(G))=\chi''(G)\leq\Delta(G)+k$. As $G$ is $(2k-1)$-connected, we have from Theorem~\ref{th2} that $T(G)$ contains a clique minor of size $\Delta(G)+k$. Thus Hadwiger's Conjecture is true for $T(G)$.
\end{proof}

By Theorem~\ref{molloy}, we have that $(C)$-\textit{TCC} is true, where $C$ is the constant in Theorem~\ref{molloy}. We thus have the following corollary.
\begin{corol}\label{cth2}
If $G$ is a $(2C-1)$-connected graph, then $T(G)$ satisfies Hadwiger's conjecture, where $C$ is the constant from Theorem~\ref{molloy}. 
\end{corol}

\section{Chromatic number and total coloring}
\begin{notation}
Given a coloring of the vertices and edges of a graph, we call a vertex $v$ or an edge $e$ which is assigned color $c$, a $c$-vertex, or a $c$-edge. Given a set of colors $C$, we call a vertex $v$ or an edge $e$ having a color from $C$ a $C$-vertex or a $C$-edge, respectively. We say that the vertex $v$ ``sees'' the color $c$, if $v$ is a $c$-vertex, or $v$ has a $c$-edge incident on it. Note that the situation that a $c$-vertex $u$ is adjacent to $v$ is not referred to by the phrase, ``$v$ sees the color $c$". For a vertex $v$, we define $S_{\varphi}(v)$ to be the set of colors that are ``seen'' by the vertex $v$ with respect to a given coloring $\varphi$.
For positive integers $i,j$ such that $i\leq j$, we denote by $[i,j]$ the set $\{i,i+1,\ldots,j\}$.
A \emph{trail} in a graph is a sequence of vertices $v_1v_2\ldots v_k$ such that for $i\in\{1,2,\ldots,k-1\}$, $v_iv_{i+1}$ is an edge of the graph, and the edges $v_1v_2,v_2v_3,\ldots,v_{k-1}v_k$ are all pairwise distinct. Note that a vertex can appear more than once on a trail. Given two subsets of edges $A$ and $B$, we say that a trail $v_1v_2\ldots v_k$ in the graph is an $(A,B)$-trail if $v_iv_{i+1}\in A$ when $i$ is odd and $v_iv_{i+1}\in B$ when $i$ is even.
\end{notation}

\begin{thm}\label{th4}
Let $G$ be a simple finite connected graph such that $\chi(G)\leq 5$. Then $\chi''(G)\leq \Delta(G)+3$.
\end{thm}
\begin{proof}
We begin by coloring $V(G)$ with colors $[5]=\{1,2,3,4,5\}$, and $E(G)$ with colors $[3,\Delta+3]=\{3,4,5,6,\ldots ,\Delta+3\}$, where $\Delta:=\Delta(G)$. Recall that such an edge coloring is possible by Theorem~\ref{viz}. Note that this coloring of $V(G)\cup E(G)$, which we shall denote by $\alpha$, need not be a total coloring. 

Clearly, in this coloring, no two adjacent vertices have the same color, and no two adjacent edges have the same color. But there could exist edges of the form $e=uv$ such that $\alpha(e)=\alpha(v)$ or $\alpha(e)=\alpha(u)$. Note that if that happens, then $\alpha(e)\in\{3,4,5\}$. We shall call such an edge a ``conflicting edge''.

We now create another coloring $\pi:V(G)\cup E(G)\rightarrow [\Delta+3]$, by slightly modifying the coloring $\alpha$ if necessary, so that the following property holds for $\pi$.
\medskip

\textbf{Property A.} There is no 5-vertex $v$ having a 3-edge, 4-edge and 5-edge incident on it.
\medskip

The coloring $\alpha$ can be transformed into such a coloring $\pi$ by repeatedly doing the following: Consider each 5-vertex $v$ having a 3-edge, 4-edge and 5-edge incident on it. Since $v$ can now see at most $\Delta-3$ colors from $[6,\Delta+3]$, there exists a color $c\in [6,\Delta+3]$ that is not seen by $v$. We assign the color $c$ to $v$. Note that no neighbours of $v$ have the color $c$ since the 5-vertices of $\alpha$ formed an independent set and the color $c$ was not assigned to any vertex in $\alpha$.

We call this coloring $\pi$ the ``original coloring'' of $G$. 
We emphasize that:
\begin{enumerate}
\renewcommand{\labelenumi}{(\roman{enumi})}
\item $\left.\pi\right|_{V(G)}$ and $\left.\pi\right|_{E(G)}$ are proper colorings of $V(G)$ and $E(G)$, respectively. 
\item $\pi$ satisfies Property~A.
\item Any conflicting edge in $\pi$ is also a conflicting edge in $\alpha$. Thus, if $e$ is a conflicting edge in $\pi$, then $\pi(e)\in\{3,4,5\}$. 
\end{enumerate}

Our goal now is to start with the coloring $\pi$ of $G$ and then recolor each conflicting edge so that it is no longer a conflicting edge in the resulting coloring. We will do so in three \textit{phases} and in each phase, we select a color $i\in\{3,4,5\}$ and recolor the conflicting $i$-edges one by one. Our recoloring strategy will rely on the following lemma.
\begin{lem}\label{shift}
Let $\varphi$ be a coloring of $V(G)\cup E(G)$ at a certain stage of the recoloring process. Let $i\in\{3,4,5\}$ and let $A_{i}=\{1,2,\ldots,i-1\}$. Suppose $\varphi$ satisfies the following properties:
\begin{enumerate}
\renewcommand{\labelenumi}{(\alph{enumi})}
\renewcommand{\theenumi}{(\alph{enumi})}
\item\label{p:vcol} $\left.\varphi\right|_{V(G)}=\left.\pi\right|_{V(G)}$,
\item $\left.\varphi\right|_{E(G)}$ is a proper edge coloring of $G$,
\item there are no conflicting $A_{i}$-edges in $G$,
\item every $A_{i}$-edge in $G$ has original color in $\{3,4,\ldots,i\}$,
\item\label{p:orgconf} every $A_{i}$-edge in $G$ that has original color $i$ has an endpoint colored $i$, that is, it was originally a conflicting $i$-edge, and
\item\label{p:orgcol} if $e'$ is a conflicting $i$-edge in $G$, then it has original color $i$, that is, $\pi(e')=\varphi(e')=i$.
\end{enumerate}
Further, let $e$ be a conflicting $i$-edge in $G$. Then $G$ can be recolored to a coloring $\psi$ so that $e$ is no longer a conflicting edge in $\psi$, and the conditions \ref{p:vcol} -- \ref{p:orgcol} hold for this new coloring $\psi$ too.
\end{lem}
\begin{proof}
Let $e=v_{0}v_{1}$, where $\varphi(e)=\varphi(v_1)=i$. By~\ref{p:orgcol}, we know that the original color of $e$ is $i$.

If there exists a color in $A_{i}\setminus (S_{\varphi}(v_{0})\cup S_{\varphi}(v_{1}))$, we can recolor $e$ with that color to get $\psi$. It is easy to verify that all the properties \ref{p:vcol} -- \ref{p:orgcol} hold for $\psi$.
 
So we shall assume that $A_{i}\subseteq S_{\varphi}(v_{0})\cup S_{\varphi}(v_{1})$. Let $X$ be the set of edges incident on $v_0$ that have original color in $[3,i]$. Since $\lvert [3,i]\rvert=i-2$ and $\pi|_{E(G)}$ is a proper edge coloring, we have $|X|\leq i-2$. Since $v_0v_1$ has original color $i$, we have $v_0v_1\in X$, which implies that $|X\setminus\{v_0v_1\}|\leq i-3$. Since $\varphi(v_{0}v_{1})=i$, it is not an $A_{i}$-edge. Thus there are at most $i-3$ $A_i$-edges in $X$. Since every $A_{i}$-edge has original color in $[3,i]$ by (d), we know that every $A_i$-edge incident on $v_0$ is in $X$. It follows that there are at most $i-3$ $A_{i}$-edges incident on $v_{0}$. Considering the possibility that $\varphi(v_{0})$ also may be in $A_{i}$, we have $|S_{\varphi}(v_{0})\cap A_{i}|\leq i-2$. This implies that there exists a color $\gamma\in A_{i}\setminus S_{\varphi}(v_{0})$. As $A_{i}\subseteq S_{\varphi}(v_{0})\cup S_{\varphi}(v_{1})$, we have $\gamma\in S_{\varphi}(v_{1})$. Therefore, as $\varphi(v_1)=i$, there exists an edge $v_{1}v_{2}$ such that $\varphi(v_{1}v_{2})=\gamma$. As $\gamma\notin S_\varphi(v_0)$, we have $v_{1}v_{2}\neq v_{0}v_{1}$.

We call an edge $f$ in $G$ an $\overline{i}$-edge if $\varphi(f)\in A_{i}\setminus\{\gamma\}$ and the original color of $f$ is $i$. By~\ref{p:orgconf}, we know that one of the endpoints of every $\overline{i}$-edge is colored $i$. Let $\Gamma$ and $\overline{I}$ be the set of all $\gamma$- and $\overline{i}$-edges in $G$, respectively. 

Now consider a maximal $(\Gamma,\overline{I})$-trail $P=v_1v_2\ldots v_k$, where $k\geq 2$, starting with the edge $v_{1}v_{2}$. It is easy to see that every edge on $P$ is an $A_{i}$-edge. Then we have the following:
\medskip

\noindent (1)\label{1} For each odd $t\in\{1,2,\ldots,k\}$, $\varphi(v_t)=i$.

Note that $\varphi(v_{1})=i$. Suppose that for some odd $t\in\{1,2,\ldots,k-2\}$, we have $\varphi(v_{t})=i$. Recall that $\left.\varphi\right|_{V(G)}=\left.\pi\right|_{V(G)}$ and that $\left.\pi\right|_{V(G)}$ is a proper vertex coloring of $G$. Then clearly $\varphi(v_{t+1})\neq i$, and since $v_{t+1}v_{t+2}$ is an $\overline{i}$-edge and an endpoint of every $\overline{i}$-edge must be colored $i$, we have $\varphi(v_{t+2})=i$.
\medskip

\noindent (2) $v_{0}P$ is a path.

Otherwise, some vertex must be repeated while traversing $v_{0}P$. Let $v$ be the first such vertex. If $v\neq v_{0}$, then it either has two edges having original color $i$ or two $\gamma$-edges incident on it; as both $\left.\varphi\right|_{E(G)}$ and $\left.\pi\right|_{E(G)}$ are proper edge colorings of $G$, this is a contradiction. Now, let $v=v_{0}$. Recall that by choice of $\gamma$, the vertex $v_{0}$ has no incident $\gamma$-edge. Then, since $v_{0}v_{1}$ has original color $i$, there must be two incident edges with original color $i$ at $v_{0}$; another contradiction. It follows that $v_{0}P$ is a path.
\medskip

Note that the last vertex $v_{k}$ on $P$ satisfies the following properties:
\begin{itemize}
    \item[(3a)] If $v_{k-1}v_{k}$ is an $\overline{i}$-edge, $\gamma \notin S_{\varphi}(v_{k})$.
    \item[(3b)] If $v_{k-1}v_{k}$ is a $\gamma$-edge, then there exists $c'\in A_{i}\setminus S_{\varphi}(v_{k})$. (Note that $c'\neq \gamma$.)
\end{itemize}
To see this, note that in the first case, we have that $k$ is odd and then by (1), $\varphi(v_{k})=i\neq\gamma$. By the maximality of $P$, we have that $v_{k}$ has no incident $\gamma$-edges. In the second case, suppose $A_{i}\subseteq S_{\varphi}(v_{k})$. This happens only if, for each $c\in A_{i}\setminus\{\varphi(v_{k})\}$, $v_{k}$ has an $c$-edge incident on it. By property (d) and the fact that $\left.\pi\right|_{E(G)}$ is a proper edge coloring, these $i-2$ edges have a distinct original color from the set $[3,i]$. Since $|[3,i]|=i-2$, we can conclude that $v_{k}$ has an incident edge $v_{k}w$ with original color $i$ and $\varphi(v_{k}w)\in A_{i}$. Is it possible that $w=v_{k-1}$? Recall that in this case, $k-1$ is odd and $v_{k-2}v_{k-1}$ is an $i$-edge (when $k=2$) or an $\overline{i}$-edge (when $k>2$). As $\left.\pi\right|_{E(G)}$ is a proper edge-coloring, $v_{k-1}v_{k}$ cannot have original color $i$ and so $w\neq v_{k-1}$. As $\varphi(v_{k-1}v_k)=\gamma$, we have $\varphi(v_kw)\neq\gamma$, and so $v_kw$ is an $\overline{i}$-edge. This means $v_{k}$ has an incident $\overline{i}$-edge outside of $P$. This violates the maximality of the trail $P$. So there exists a color $c'\in A_{i}\setminus S_{\varphi}(v_{k})$.
\medskip

Now we recolor the edges on $v_{0}P$ by the following color shift strategy : 
\begin{itemize}
    \item  For each $1\leq t\leq k-1$, assign $\varphi(v_{t}v_{t+1})$ to the preceding edge $v_{t-1}v_{t}$ by setting $\psi(v_{t-1}v_{t}):=$ $\varphi(v_{t}v_{t+1})$.
    
    \item If the last edge $v_{k-1}v_{k}$ on $P$ is an $\overline{i}$-edge, set $\psi(v_{k-1}v_{k}):=$ $\gamma$.
    
    \item If $v_{k-1}v_{k}$ is a $\gamma$-edge, pick a color $c'\in A_{i}\setminus S_{\varphi}(v_{k})$ and set $\psi(v_{k-1}v_{k}):=c'$. Such a $c'$ exists by (3b). 
\end{itemize}

Finally, set $\psi(v):=\varphi(v)$, for each vertex $v\in V(G)$, and $\psi(e):=\varphi(e)$, for each edge $e\in E(G)\setminus E(v_{0}P)$.

It is easy to see that properties (a), (d), (e), and (f) hold for this new coloring $\psi$: We have not recolored any vertex, therefore (a) holds; Other than $v_0v_1$, all the edges that we have recolored are $A_{i}$-edges (with respect to $\varphi$) and they remain $A_{i}$-edges with respect to $\psi$ as well; The only new $A_{i}$-edge with respect to $\psi$ is $e=v_{0}v_{1}$. This is consistent with (d) and (e); the set of conflicting $i$-edges with respect to $\psi$ is the set of conflicting $i$-edges with respect to $\varphi$ minus $v_{0}v_{1}$, therefore (f) holds. We shall now proceed to show that the color shift strategy produces no new conflicts, ensuring that $\psi$ satisfies properties (b) and (c) also.

We first show that for any $t\in\{1,2,\ldots,k-1\}$, we have $\psi(v_{t-1}v_t)\neq\psi(v_tv_{t+1})$. For $t\leq k-2$, we know that $\psi(v_{t-1}v_t)=\varphi(v_tv_{t+1})$, $\psi(v_tv_{t+1})=\varphi(v_{t+1}v_{t+2})$, and that $\varphi(v_tv_{t+1})\neq\varphi(v_{t+1}v_{t+2})$ (as $\varphi$ satisfies (b)). It follows that $\psi(v_{t-1}v_t)\neq\psi(v_tv_{t+1})$. Next, if $t=k-1$, then $\psi(v_{k-2}v_{k-1})=\varphi(v_{k-1}v_k)$ and by (3a) and (3b), $\psi(v_{k-1}v_k)\notin S_\varphi(v_k)$. Since $\varphi(v_{k-1}v_k)\in S_\varphi(v_k)$, this means that $\psi(v_{k-2}v_{k-1})\neq\psi(v_{k-1}v_k)$. Thus in $\psi$, no two consecutive edges of $v_0P$ have the same color.

Since we are recoloring only the edges on $v_{0}P$, if (b) or (c) is violated in $\psi$, there must be an edge on $v_{0}P$ such that it conflicts either with one of its end-points or with an adjacent edge. Let $t$ be any index such that $v_{t}v_{t+1}$ is such an edge with respect to $\psi$. Let $c=\psi(v_{t}v_{t+1})$. We will say that the edge $v_{t}v_{t+1}$ has a conflict at the end-point $v_{t+1}$, if either $\psi(v_{t+1})=c$ or $v_{t+1}$ has an incident edge $v_{t+1}u\neq v_tv_{t+1}$ with $\psi(v_{t+1}u)=c$. Otherwise, the conflict has to be at the end-point $v_{t}$.

We claim that $v_{t}v_{t+1}$ cannot have a conflict at the end-point $v_{t+1}$. If $t=k-1$, this is obvious from (3a) and (3b). For $t\leq k-2$, $\psi(v_{t}v_{t+1})=\varphi(v_{t+1}v_{t+2})=c$. Suppose that $\psi(v_{t+1})=c$. Then since we are recoloring only edges, we have $\varphi(v_{t+1})=c$, which contradicts the fact that $\varphi$ satisfies (c). Next, suppose that there exists an edge $v_{t+1}u\neq v_tv_{t+1}$ such that $\psi(v_{t+1}u)=c$. Since no two consecutive edges of $v_0P$ have the same color in $\psi$, we know that $u\neq v_{t+2}$. Then $v_{t+1}u$ is not an edge of $v_0P$ which means that $\varphi(v_{t+1}u)=\psi(v_{t+1}u)=c$, which implies that there are two edges with color $c$ incident at $v_{t+1}$ in $\varphi$. This contradicts the fact that $\varphi$ satisfied (b). Therefore, in $\psi$, the edge $v_tv_{t+1}$ cannot have a conflict at the end-point $v_{t+1}$.

Hence, $v_{t}v_{t+1}$ has a conflict at the end-point $v_{t}$. We consider the two possible cases separately.
\medskip

\noindent\textbf{Case 1:} $\psi(v_{t}v_{t+1})=\gamma$.
\medskip

In this case the edge $v_{t}v_{t+1}$ is an $\overline{i}$-edge (if $t>0$) or $i$-edge (if $t=0$) before the recoloring. Note that this means that $t+1$ is odd. If $t=0$, then by the choice of $\gamma$, we have $\gamma\notin S_{\varphi}(v_{0})$ and thus the edge $v_0v_1$ has no conflict at the end-point $v_{0}$ in $\psi$. For $t\geq 2$, since $\varphi(v_{t-1}v_{t})=\gamma$, and $\varphi$ satisfied (c) and (b), we know that $\psi(v_t)=\varphi(v_t)\neq\gamma$ and that no edge incident at $v_t$ other than $v_{t-1}v_t,v_tv_{t+1}$ has the color $\gamma$ in $\psi$. Since in $\psi$, no two consecutive edges on $v_0P$ have the same color, $\psi(v_{t-1}v_t)\neq\gamma$, and so we can conclude that $v_tv_{t+1}$ has no conflict at the end-point $v_t$ in $\psi$ in this case.
\medskip

\noindent\textbf{Case 2:} $\psi(v_tv_{t+1})\neq\gamma$.
\medskip

In this case, $(\varphi(v_{t+1}v_{t+2})=\psi(v_tv_{t+1}))\neq\gamma$, which implies that $v_{t+1}v_{t+2}$ is an $\overline{i}$-edge before the recoloring. Then $t$ is odd and $\varphi(v_{t}v_{t+1})=\gamma$.
By (1), $\psi(v_{t})=\varphi(v_{t})=i$. Since $\psi(v_tv_{t+1})\in A_i$ and $v_tv_{t+1}$ has a conflict at the end-point $v_{t}$, this implies that there exists some edge $v_{t}x\neq v_tv_{t+1}$ such that $\psi(v_tx)=\psi(v_tv_{t+1})$. As no two consecutive edges on $v_0P$ have the same color in $\psi$, we have that $x\neq v_{t-1}$. Then $\varphi(v_tx)=\psi(v_tx)=\psi(v_tv_{t+1})$. Thus, the three edges $v_{t-1}v_{t}$, $v_{t}x$, and $v_{t}v_{t+1}$ are all $A_{i}$-edges with respect to $\varphi$, and by property (d) have original colors in $[3,i]$. Since $\left.\pi\right|_{E(G)}$ is a proper edge-coloring, all these three edges should have distinct colors with respect to $\pi$. This can happen only if $\lvert [3,i]\rvert\geq 3$, which happens only if $i=5$. Then with respect to $\pi$, $v_{t}$ is a $5$-vertex with a $3$-edge, a $4$-edge, and a $5$-edge all incident on it. This contradicts the fact that $\pi$ satisfies Property~A.
\medskip

Thus, we conclude that in $\psi$, no edge on $v_0P$ has conflict at either of its end-points, which also implies that $e=v_0v_1$ is not a conflicting edge in $\psi$. Thus $\psi$ satisfies (b), and since $e$ is the only new $A_{i}$-edge in $\psi$, property (c) holds for $\psi$. This completes the proof.
\end{proof}

With the help of the above lemma, we can start with the original coloring $\pi$ and proceed with the recoloring in three \textit{phases}. In each phase, we recolor all the conflicting $i$-edges by invoking Lemma \ref{shift} first with $i=3$ in phase I, then $i=4$ in phase II, and finally $i=5$ in phase III. Note that after the phase corresponding to $i$, there are no conflicting $i$-edges in $G$. Therefore, after the three phases there are no conflicts in $G$. Thus, we have a total coloring of $G$ using colors from $\{1,2,3,\ldots,\Delta+3\}$ implying that, $\chi''(G)\leq \Delta+3$.
\end{proof}

\section{Open problems}

\begin{enumerate}
    \item Corollary \ref{cth2} states that Hadwiger's conjecture is true for graphs with high vertex-connectivity, that is, at least $C'$ where $C'$ is a constant.
    
    It is interesting to see whether a similar statement can be made with respect to edge-connectivity instead of vertex-connectivity. Note that Theorem \ref{nw} uses only edge-connectivity. Hence, it would be natural to expect Corollary \ref{cth2} to be expressed in terms of edge-connectivity. Unfortunately, we could not find a way to do it.
    
    Of course, it would be most desirable to prove Hadwiger's conjecture for all total graphs dropping any constraints. 
    
    \item In Theorem \ref{tcc3}, we show that if  weak \textit{TCC} holds for all graphs, then Hadwiger's conjecture is true all total graphs. However, improving this result to get the following remains challenging.
    \begin{center}
    ``If $(k)$-\textit{TCC} is true for all graphs,  then Hadwiger's conjecture is true for all total graphs.''
    \end{center}
    for some fixed positive integer $k>3$. The easiest case to try out is when $k=4$ and we think that this case itself would be interesting for it may require a new idea compared to the one used for proving the $k=3$ case.
    \item Since we could not prove Hadwiger's conjecture for all total graphs, it is interesting to try to prove Hadwiger's conjecture for total graphs of specialized graph classes for which weak \textit{TCC} is not yet known to hold.
    \item It remains hard to prove weak \textit{TCC} for the general case. In Theorem \ref{th4}, we proved it for $5$-colorable graphs. The current best upper bound on the total chromatic number of $k$-colorable graphs when $6\leq k\leq 8$ is $\Delta+4$. Proving weak \textit{TCC} for $k$-colorable graphs even for small values of $k$, say $k=6,7,\ldots$, remains open.
    \item In the proof of Theorem \ref{th4}, $\chi'(G)$ colors are used to do the initial proper edge coloring and an additional two colors are used to get a proper vertex coloring. This suggests the question whether it is possible, when $\chi'(G)=\Delta(G)$ (that is, when $G$ is \textit{class I}), to obtain the bound $\chi''(G)\leq\Delta(G)+2$, instead of the $\Delta(G)+3$ bound that we proved in Theorem \ref{th4}. Unfortunately, our proof fails to achieve this since even when dealing with \textit{class I} graphs, we need one extra color, the $(\Delta+3)$-th color, to establish the Property~A for the original coloring. Property~A is crucial in phase III of the recoloring process, that is, when conflicting $5$-edges are recolored (see proof of Lemma \ref{shift}). 
    
    It would be interesting to prove, $\chi''(G)\leq \Delta(G)+2$ when $G$ is a \textit{class I} $5$-colorable graph, possibly by some clever tweaking of the proof of Theorem \ref{th4}. 
\end{enumerate}

\section* {Acknowledgements}

The first author was partially supported by the fixed grant scheme
 SERB-MATRICS project number MTR/2019/000790.
The second author was partially supported by the fixed grant scheme
 SERB-MATRICS project number MTR/2018/000600.

\bibliographystyle{unsrt}

\begin{thebibliography}{30}

\bibitem{appel}
K. Appel and W. Haken. :
\textit{Every planar map is four colorable. Part I: Discharging}. Illinois J. Math.,
21(3):429{-}490, 1977.

\bibitem{appel2}
K. Appel, W. Haken, and J. Koch. :
\textit{Every planar map is four colorable. Part II: Reducibility}. Illinois J. Math., 21(3):491{-}567, 1977.

\bibitem{behzad}
M. Behzad: 
\textit{Graphs and Their Chromatic Numbers}, Ph.D. thesis, Michigan State
University (1965).

\bibitem{bcc}
M. Behzad, G. Chartrand, J.K. Cooper : 
\textit{The color numbers of complete graphs}, J. London Math. Soc. 42, 226-228 (1967)

\bibitem{belkale}
N. Belkale and L. S. Chandran. :
\textit{Hadwiger's conjecture for proper circular arc graphs}. European J. Combin., 30(4):946{-}956, 2009.

\bibitem{bh}
B. Bollob\'as and A. Harris : 
\textit{List-colorings of Graphs. Graphs and Combinatorics}, (1) 115-127
(1985)

\bibitem{borodin}
O. V. Borodin :
\textit{On the total coloring of planar graphs}, J. Reie Angew. Math., 394: 180-185, 1989.

\bibitem{chandran}
L. S. Chandran, D. Issac, and S. Zhou : 
\textit{Hadwiger’s conjecture for squares of 2-trees}. European Journal of Combinatorics. 76. 159-174, (2019).

\bibitem{chew}
K. H. Chew : 
\textit{Edge colorings and total colorings of graphs and multigraphs}, Doctoral Thesis, University of New South Wales, 1994.

\bibitem{chud}
M. Chudnovsky and A. O. Fradkin :
\textit{Hadwiger's conjecture for quasi-line graphs}. J. Graph Theory, 59(1):17{-}33, 2008.

\bibitem{diestel}
R. Diestel :
\textit{Graph Theory}. Fourth Vol. 173. Heidelberg; New York: Springer, 2010. 

\bibitem{ert}
P. Erd\"os, A. L. Rubin, and H. Taylor :
\textit{Choosability in graph}, Cong. Numer. 26, 125-157, 1979.

\bibitem{had}
H. Hadwiger : 
\textit{\"Uber eine Klassifikation der Streckenkomplexe}, Vierteljschr. Naturforsch. Ges. Z\"urich, 88: 133–143, 1943.

\bibitem{hind} 
H. R. Hind :
\textit{An Upper Bound for the Total Chromatic Number}, Graphs and Combinatorics, 6 (1990), 153-159. 

\bibitem{kostochka} 
A. V. Kostochka :
\textit{The total chromatic number of any multigraph with maximum degree five is at most seven}, Discrete Math., 162(1-3): 199-214, 1996.

\bibitem{kostochka2}
A. V. Kostochka :
\textit{The total coloring of a multigraph with maximum degree 4}, Discrete Math., 17(2): 161-163, 1977.

\bibitem{kostochka3}
A. V. Kostochka : 
\textit{An analogue of Shannon's estimate for complete colorings} (Russian). Diskret.
Analiz. 30, 13-22 (1977).

\bibitem{li}
D. Li and M. Liu :
\textit{Hadwiger's conjecture for powers of cycles and their complements}. European J.
Combin., 28(4):1152{-}1155, 2007.

\bibitem{molloyreed}
M. Molloy and B. Reed :
\textit{A bound on the total chromatic number}. Combinatorica, 18(2):241{-}280, 1998.

\bibitem{reed}
B. Reed and P. Seymour :
\textit{Hadwiger's conjecture for line graphs}. European J. Combin., 25(6):873{-}876, 2004.

\bibitem{rst}
N. Robertson, P. Seymour, and R. Thomas :
\textit{Hadwiger's conjecture for $K_{6}$-free graphs}. Combinatorica, 13(3):279{-}361, 1993.

\bibitem{rosenfeld}
M. Rosenfeld : 
\textit{On the total coloring of certain graphs}, Israel J. Math., 9: 396-402. 1971.

\bibitem{sanchez}
A. S\'anchez-Arroyo :
\textit{Determining the total coloring number is NP-hard}, Discrete math. 78(1989), 315-319.

\bibitem{sanderzhao}
D. P. Sanders and Y. Zhao :
\textit{On total $9$-coloring planar graphs of maximum degree seven}, J. Graph Theory, 31(1): 67-73, 1999.

\bibitem{sz}
D. P. Sanders and Y. Zhao :
\textit{Planar Graphs of Maximum Degree Seven are Class $I$}, J. Comb. Theory B. 83, 201-212, 2001.

\bibitem {scottseym}
Alex Scott and Paul Seymour :
\textit {Survey of $\chi$-boundedness},  J. Graph. Theory. 95(3), 473-504, 2020.



\bibitem{vizing}
V. Vizing :  
\textit{Some unsolved problems in graph theory}, Russian Math Surveys, 23
(1968), 125-141.

\bibitem{vijay} 
N. Vijayaditya :
\textit{On total chromatic number of a graph}, J London Math Soc. 3
(1971), 405-408.

\bibitem{wagner}
K. Wagner :
\textit{\"Uber eine Eigenschaft der ebenen Komplexe}. Math. Ann., 114(1) : 570{-}590, 1937.

\bibitem{wood}
D. Wood, G. Xu, and S. Zhou :
\textit{Hadwiger's conjecture for $3$-arc graphs}. Electronic J. Combin., 23(4):\#P4.21, 2016.

\bibitem{xu}
G. Xu and S. Zhou :
\textit{Hadwiger's conjecture for the complements of Kneser graphs}. J. Graph Theory,
84(1):5{-}16, 2017

\bibitem{yap}
H. P. Yap : \textit{Total colorings of Graphs}, Lecture Notes in Mathematics 1623, Springer, New York, 1996.

\end{thebibliography}

\end{document}